\documentclass[12pt]{article}
\usepackage{a4wide}
\usepackage{amsmath,amsthm}
\usepackage{amsfonts}
\usepackage{amssymb}
\usepackage{url}
\usepackage{xspace}
\usepackage{graphicx, epstopdf}
\usepackage{pdfsync}
\usepackage{pdflscape}
\usepackage{framed}
\usepackage[x11names,usenames,dvipsnames,svgnames,table]{xcolor}
\usepackage[hidelinks]{hyperref}

\numberwithin{equation}{section}

\newtheorem{theorem}{Theorem}[section]
\newtheorem{lemma}{Lemma}[section]
\newtheorem{proposition}{Proposition}[section]
\newtheorem{corollary}{Corollary}[section]

\newtheorem{assumption}{Assumption}[section]

\theoremstyle{definition}

\theoremstyle{remark}

\DeclareMathOperator*{\argmin}{argmin}
\DeclareMathOperator{\trace}{tr}
\DeclareMathOperator{\val}{val}
\DeclareMathOperator{\diag}{diag}

\newcommand{\bv}{\mathbf{v}}
\newcommand{\bx}{\mathbf{x}}
\newcommand{\bg}{\mathbf{g}}
\newcommand{\cJ}{\mathcal{J}}
\newcommand{\real}{\mathbb{R}}
\newcommand{\hardstepname}{standard\xspace}

\title{An optimal variant of Kelley's cutting-plane method\thanks{This research was partially supported by the Israel Science Foundation under ISF grant no. 998-12.}}
\author{Yoel Drori$^\dagger$ \and Marc Teboulle\thanks{School of Mathematical Sciences, Tel-Aviv University, Ramat-Aviv 69978,
Israel ({\tt dyoel@post.tau.ac.il, teboulle@math.tau.ac.il})}}
\date{\today}

\begin{document}
\maketitle

\begin{abstract}
We propose a new variant of Kelley's cutting-plane method
for minimizing a nonsmooth convex Lipschitz-continuous function over the Euclidean space.
We derive the method through a constructive approach and prove that it attains the optimal rate of convergence
for this class of problems.

\end{abstract}

\section{Introduction}
In this paper, we focus on unconstrained nonsmooth convex minimization problems,
where information on the objective can only be gained
through a first-order oracle, which returns the value of the objective and an element in its subgradient at any point in the problem's domain.
Problems of this type often arise in real-life applications
either as the result of a transformation that was applied on a problem
(such as Benders' decomposition~\cite{benders1962partitioning})
or by some inherent property of the problem
(e.g., in an eigenvalue optimization problem).

One of the earliest and most fundamental methods for solving nonsmooth convex problems
is Kelley's cutting plane method (or, the Kelley method, for short),
which was introduced by Kelley in~\cite{kelley1960cutting} and also independently by Cheney and Goldstein~\cite{cheney1959newton}.
The method maintains a polyhedral model of the objective,
and at each iteration
updates this model according to the first-order information
at a point where the model predicts that the objective is minimal.
Despite the elegant and intuitive nature of this method,
the Kelley method suffers from very poor performance, both in practice and in theory~\cite{nemi-yudi-book83}.
The source of the poor performance seems to be the instability of the solution of the subproblems,
where the iterates of the method tend to be far apart and at locations where the accuracy of the model is poor.







The main objective of this work is to present a new method for minimizing
a nonsmooth convex Lipschitz-continuous function over the Euclidean space,
which is surprisingly similar to the Kelley method,
yet attains the optimal rate of convergence for this class of problems.
We derive this method and its rate of convergence through a constructive approach which
further develops the recent framework we introduced in \cite{drori2013performance}.
In the later work, a novel approach was developed to derive new complexity bounds for a
broad class of first order schemes for {\em smooth} convex minimization. The approach
is based on the observation that
the efficiency estimate of a  method can be formulated  
as an optimization problem
and once this is done, it is possible to optimize the parameters of the method to achieve the best possible efficiency estimate
(this can be viewed as some kind of a ``meta-optimization'' approach, where we optimize
the parameters of an optimization method). Very recently, these results were further analyzed  in \cite{kim2014optimized} to derive optimized first-order methods for smooth convex minimization.


Although the main contribution of this work is entirely theoretical, it should be noted that the resulting method also offers
some practical advantages over existing bundle methods.
One of the main advantages is that
the method allows the implementation to choose at each iteration between two types of steps:
a ``\hardstepname'' step, which, as in all bundle methods, requires solving an auxiliary convex optimization program,
and an ``easy'' step which involves only a subgradient step with a predetermined step size.
The efficiency estimate of the method remains valid regardless of the choices a specific implementation makes,
thereby allowing the implementation to find a balance between accuracy and speed
(without performing aggregation on the iterates, which affects the accuracy of the model).

One limitation of the method is that it requires choosing the number of iterations to be performed in advance.
However, this limitation is not severe since
the ``\hardstepname'' steps provide as a by-product
a bound on the worst-case absolute inaccuracy at the \emph{end} of the method's run,
hence once the desired accuracy has been achieved,
the implementation can choose to perform only ``easy'' steps thereby quickly ending the execution of the method.


\paragraph{Literature} The first successful approach for overcoming the instability in the Kelley method,
known as the bundle method, was introduced by Lemar\'echal~\cite{lemarechal1975extension} and
also independently by Wolfe~\cite{wolfe1975method}.
In the bundle approach,
the instability in the Kelley method
is tackled
by introducing a regularizing quadratic term in the objective,
thereby forcing the next iterate to remain in close proximity to the previous iterates, where the model is more accurate.
The bundle approach proved to be very fruitful, and yielded many variations on the idea,
see for instance \cite{auslender1987numerical, kiwiel1990proximity, lemarechal1997variable} and references therein.
The bundle method and its variants also proved to perform very well in practice, however,
a theoretical rate of convergence is not available for most variants,
and for the variants where a rate of convergence was established,
it was shown to be suboptimal~\cite{kiwiel2000efficiency}.



Another fundamental approach
is the level bundle method,
introduced by Lemar\'echal et al.~\cite{lemarechal1995new}.
The idea behind this approach is that the level sets of the polyhedral model of the objective
are ``stable'', and therefore they should be used instead of the complete model.
Building on this idea,
at each iteration the method performs a projection of the previous iterate on a carefully selected level set of the model,
then updates the model according to the first-order information at the resulting point.
Several extensions to the method were proposed, including
a restricted memory variant~\cite{kiwiel1995proximal} and a variant for handling non-Euclidean metrics~\cite{ben2005non}.
The method was shown to possess an optimal rate of convergence,
however, note that the constant factor in the bound is not optimal, and leave room for improvement.



Finally, let us mention that
quite a few additional approaches were proposed.
Among them are
trust-region bundle methods~\cite{schramm1992version}
and the bundle-newton method~\cite{lukvsan1998bundle}, where the objective is approximated by a combination of
polyhedral and quadratic functions.
For a comprehensive survey, 
we refer the reader to~\cite{makela2002survey}.

\paragraph{Outline of the paper.} The paper is organized as follows.
In Section~\ref{S:kelleyvariant}, we present the new Kelley-Like Method (KLM), 
and state our main result: an optimal rate of convergence (Theorem~\ref{T:main}).
The motivation for the method and our approach is described in Section~\ref{S:approach}.
In Sections~\ref{S:tracub}--\ref{S:proofmt}, we
provide a detailed description of the construction of the proposed method and prove its rate of convergence.
Section~\ref{S:numerical}  numerically illustrates the potential benefit of the standard steps by comparing 
between the different variants of the method.
We conclude the main body of the work, in Section~\ref{S:concluding}, where we discuss some additional cases where the approach presented here is applicable.
Finally, in Appendix~\ref{S:appendix}, we give a new lower-complexity bound for the
class of convex and Lipschitz-continuous minimization problems, which shows
that the KLM attains the best possible rate of convergence for this class of problems.

\paragraph{Notation.} For a convex function $f$, its subgradient at $x$ is denoted by
$\partial f(x)$ and we use $f'(x)$ to denote some element in $\partial f(x)$.
We also denote $f^\ast = \min_x f(x)$ and $x^\ast =x^\ast_f\in \argmin_x f(x)$.
The Euclidean norm of a vector $x$ is denoted as $\|x\|$.
We use $e_i$ for the $i$-th canonical basis vector, which consists of all zero components,
except for its $i$-th entry which is equal to one. For an optimization problem $(P)$, $\val(P)$
stands for its optimal value.
For a symmetric
matrix $A$, $A\succeq 0$ means $A$ is
positive semidefinite (PSD).

To simplify some expressions, we often write $A \succeq 0$ for
a non-symmetric matrix $A$: this should be interpreted as $\frac{1}{2}(A+A^T)\succeq 0$.

\section{The Algorithm and its Rate of Convergence}\label{S:kelleyvariant}

In this section we present our main results, namely the new proposed algorithm and its rate of convergence.

\subsection{The Algorithm: a Kelley-Like Method (KLM)}\label{SS:klm}
Consider the minimization problem $\min\{f(x):x\in\mathbb{R}^p\}$,
where $f$ is $f\in C_L(\mathbb{R}^p)$ (i.e., $f:\mathbb{R}^p\rightarrow \mathbb{R}$ is Lipschitz-continuous with constant $L>0$)
and convex.
The method described below assumes that
$x^\ast \in \argmin_x f(x)$
is located inside a ball of radius $R>0$ around a given point $x_0\in\mathbb{R}^p$
and requires knowing in advance the number of iterations to be performed, $N$.
The method proceeds as follows:

\begin{oframed}
\noindent\textbf{Algorithm KLM}
\paragraph{Initialization:}
(The zeroth iteration.)
Set
\begin{align*}
	x_1:= x_0, \
	s := 0,\
	\tau := 1,\ \text{and }
	\mu := \frac{R}{L\sqrt{N}}.
\end{align*}


\paragraph{Iteration \#$M$:} At the $M$th iteration ($1\leq M\leq N-1$), the method \emph{arbitrarily} chooses between two types of steps:

In the first type (the ``\hardstepname step''), we set $m\in \argmin_{1\leq i\leq M} f(x_i)$ and solve
\begin{align*}
	    (B_M)\quad \max_{y\in \mathbb{R}^p, \zeta,t\in \mathbb{R}} &\ f(x_m)-t \\
		\text{s.t. }
			& f(x_i)+\langle y-x_i,f'(x_i)\rangle\leq t, \quad i=1,\dots,M,\\
			& f(x_m)-L\zeta\leq t, \\
			& \|y-x_0\|^2+(N-M)\zeta^2\leq R^2.
\end{align*}
Let $y^\ast$, $\zeta^\ast$ and $t^\ast$ be
an optimal solution to the primal variables of problem $(B_M)$,
and let $\beta^\ast$ be the optimal dual multiplier that corresponds to the constraint
$f(x_m)-L\zeta\leq t$.
The step then proceeds by setting
\[
	\text{(\hardstepname step)}\quad x_{M+1}:=y^\ast,
\]
and updating
\begin{align*}
	s := M,\
	\tau := \beta^\ast,\
	\mu := \frac{\zeta^\ast}{L}.
\end{align*}

The second type of step (the ``easy step'') is a subgradient step with
the previously selected step size $\mu$:
\[
	\text{(easy step)}\quad x_{M+1} := x_M - \mu f'(x_M).
\]

\paragraph{Output:} The output 
is given by a convex combination of the best step from the first $s$ steps
and the ergodic combination of the last $N-s$ steps:
\[
	\bar x_N := (1- \tau) x_m  + \frac{\tau}{N-s} \sum_{j=s+1}^{N} x_j ,
\]
here $m\in \argmin_{1\leq i\leq s} f(x_i)$.
\end{oframed}

Note that if the method chooses to perform an ``easy'' step at every iteration,
it simply reduces to the subgradient method with a constant step size.
Also note that the ``\hardstepname'' step shares the computational simplicity of the
main step in the Kelley method (cf.\ next section), where the two iteration rules
differ only in the introduction of the optimization variable $\zeta$ and in the inclusion of the second constraint in $(B_M)$.

\subsection{An Optimal Rate of Convergence for KLM}

We now state the 
efficiency estimate of the method,
which shows that the new method is optimal for the class of nonsmooth minimization with convex and Lipschitz-continuous
functions (see~Appendix~\ref{S:appendix} and also \cite{nemi-yudi-book83, nest-book-04}).

\begin{theorem}\label{T:main}
Suppose $\bar x_N$ was generated by Algorithm~KLM, then $f(\bar x_N)-f^\ast \leq \frac{LR}{\sqrt{N}}$.
Furthermore, suppose the method performed ``\hardstepname'' steps at iterations $s_1,\dots, s_k$,
where $s_1<\dots<s_k$,
then
\begin{equation}\label{E:minbound}
	f(\bar x_N)-f^\ast \leq \val (B_{s_k})\leq \dots \leq \val (B_{s_1}) \leq \frac{LR}{\sqrt{N}}.
\end{equation}
\end{theorem}

Note that although the rate of convergence is of same order as for
the level bundle method~\cite{lemarechal1995new}, which to the best of our knowledge has the best known efficiency estimate for a bundle method,
the constant term here is smaller by a factor of two. Hence, the proposed method requires a quarter of the steps
in order to the reach the same worst-case absolute inaccuracy.

The rest of this paper is devoted to the detailed construction of the proposed Algorithm KLM
and to the proof of Theorem~\ref{T:main}.


\section{Motivation}\label{S:approach}

\subsection{A New Look at the Kelley Method}

Consider the problem
\[
	\min_{x\in \mathbb{R}^p} f(x),
\]
where $f(x)$ is convex, nonsmooth, and Lipschitz-continuous with constant $L$.
For a given set of trial points, $\cJ_M:=\{(x_j,f(x_j),f'(x_j))\}_{j=1}^M$,
denote by $f_M(x)$ the polyhedral model of the function $f$, defined by
\begin{equation}\label{D:fmax}
	f_M(x) = \max\{ f(x_j)+\langle f'(x_j),x-x_j\rangle \mid 1\leq j \leq M\}.
\end{equation}

Assuming that $x^\ast_f \in \argmin_x f(x)$ lies inside a compact set,
which we take here as
$\{x:\|x - x_0\|\leq R\}$ for some $x_0\in \mathbb{R}^p$ and $R>0$,
the Kelley method chooses the next iterate, $x_{M+1}$, by solving
\[
	(\text{Kelley})\quad x_{M+1}\in\argmin_{\|x-x_0\|\leq R} f_M(x).
\]
Alternatively, we can write the previous rule as
the following functional optimization problem:
\begin{align*}
	(\text{Kelley}')\quad x_{M+1}  \in \argmin_{\|x-x_0\|\leq R}\ &\min_{\varphi\in C_L, \varphi \text{ is convex}} \ \varphi(x) \\
	\text{s.t. }
		& \varphi(x_i)=f(x_i), \quad i=1,\dots,M,\\
		& f'(x_i)\in \partial \varphi(x_i), \quad i=1,\dots,M, \\
		& \|x^\ast_\varphi-x_0\| \leq R,
\end{align*}
where the two formulations are equivalent since the solution to the inner minimization problem reduces
exactly to $f_M$ inside the ball $\|x-x_0\|\leq R$.

The well-known inefficient nature of the method is now apparent:
the method chooses the next iterate as one that minimizes the \emph{best-case function value},
which is not a natural strategy when we are interested in obtaining a bound on
the \emph{worst-case absolute inaccuracy}, $f(x_{M+1})-f^\ast$. This motivates us to consider the following alternative strategy.


\subsection{The Proposed Approach}
Since we are interested in deriving a bound on the worst-case behavior of the absolute inaccuracy,
a natural approach,
given a set of trial points, $\cJ_M:=\{(x_j,f(x_j),f'(x_j))\}_{j=1}^M$,
might be to choose the next iterate in a way that the worst-case absolute inaccuracy is minimized, i.e.,
\begin{align*}
	x_{M+1} \in \argmin_{x\in \mathbb{R}^p} \max_{\varphi\in C_L, \varphi \text{ is convex}} &\ \varphi(x)-\varphi^\ast \\
	\text{s.t. }
		& \varphi(x_i)=f(x_i), \quad i=1,\dots,M,\\
		& f'(x_i)\in \partial \varphi(x_i), \quad i=1,\dots,M, \\
		& \|x^\ast_\varphi-x_0\| \leq R.
\end{align*}
It appears, however, that this \emph{greedy} approach
forces the resulting iterates to be too conservative.
In fact, numerical tests show that in some cases the sequence generated by this
approach does not even converge to a minimizer of $f$!

We therefore
take a \emph{global} approach and attempt to minimize a bound on the worst-case behavior of the entire sequence,
i.e., instead of choosing only the next iterate $x_{M+1}$,
given some $N>M$,
we look for a sequence $x_{M+1},\dots,x_N$ for which the absolute inaccuracy at the last iterate, $x_N$, is minimized.
In order to accomplish this, we need to
assume  some form of structure on
the sequence $\{x_1,\dots,x_N\}$.

Let $\{v_1,\dots,v_r\}$ be an orthonormal set that spans $\{f'(x_1),\dots,f'(x_{M}), x_1-x_0,\dots, x_M-x_0\}$.
Hereafter, we
consider
sequences $x_{M+1},\dots,x_N$ that are generated
according to a first-order method of the form
\begin{equation}\label{E:mainalg}
	x_{i} = x_0 + \sum_{k=1}^{i-1} h^{(i)}_{1,k} (x_k-x_0) - \sum_{k=1}^{r} h^{(i)}_{2,k} v_k - \sum_{k=M+1}^{i-1} h^{(i)}_{3,k} f^\prime(x_k), \quad i=M+1,\dots,N,
\end{equation}
for step sizes $h^{(i)}_{j,k}\in\mathbb{R}$ that depend only on the data available at the current stage (i.e., $L$, $R$  and $\cJ_M$).
Note that the
first summation
is redundant here
and can be expressed using the other terms,
however, including it
will significantly simplify the following analysis.

For sequences of this form,
given $h=(h^{(i)}_{j,k})$,
the worst-case absolute inaccuracy at $x_N$ is, by definition, the solution to
\begin{align*}
	P_M(h):= \max_{\varphi\in C_L, \varphi \text{ is convex}} \ &\varphi(x_N)-\varphi^\ast \\[-5pt]
	\text{s.t. }
		& x_{i} = x_0 + \sum_{k=1}^{i-1} h^{(i)}_{1,k} (x_k-x_0) - \sum_{k=1}^{r} h^{(i)}_{2,k} v_k - \sum_{k=M+1}^{i-1} h^{(i)}_{3,k} \varphi^\prime(x_k), \\&\qquad\qquad i=M+1,\dots,N,\\
		& \varphi(x_i)=f(x_i), \quad i=1,\dots,M,\\
		& f'(x_i)\in \partial \varphi(x_i), \quad i=1,\dots,M, \\
		& \|x^\ast_\varphi-x_0\| \leq R.
\end{align*}
Therefore, the problem of finding step sizes $h$
such that the worst-case absolute inaccuracy at $x_N$ is minimized
can be expressed by
\[
	(P_M) \quad \min_{h} P_M(h).
\]

Note that obtaining an optimal solution for $(P_M)$ is not necessary.
Indeed, suppose that for any $h$ we can find a (preferably easy) upper bound $Q_M(h)$ for $P_M(h)$,
then it follows that
\[
	f(x_N)-f^\ast\leq P_M(h)\leq Q_M(h),
\]
hence a method with a ``good'' worst-case absolute inaccuracy might be found by minimizing $Q_M(h)$ with respect to $h$ instead of $P_M(h)$.
The analysis developed in the forthcoming two sections show how to achieve this, and  serves two main goals:
\begin{itemize}
\item Derive a tractable upper-bound for the worst-case absolute inaccuracy expressed via problem $(P_M)$.
\item Show that the derivation of this bound leads itself to the construction of Algorithm~KLM.
\end{itemize}

\section{A Tractable Upper-Bound for $(P_M)$}\label{S:tracub}

Problem $(P_M(h))$ (and hence problem $(P_M)$) is a difficult abstract optimization problem in infinite dimension through the functional
constraint on $\varphi$.
Inspired by the approach developed in~\cite{drori2013performance},
we start the reformulation of the problem by deriving a finite dimensional relaxation (\S\ref{S:finitedimrelax}),
followed by
an SDP relaxation
for the inner maximization problem (\S\ref{S:sdprelaxation}).
We then consider the resulting minimax problem and show how,
using duality, linearization, and the matrix completion theorem,
it can be transformed into a tractable problem (\S\ref{S:takingdual}--\ref{S:tightconvex}).

\subsection{A Finite Dimensional Relaxation of $(P_M)$}\label{S:finitedimrelax}

To relax $(P_M)$ into a finite dimensional problem, we need to tackle the constraint
``$\varphi\in C_L,\ \varphi \text{ is convex}$'',
which states that for all $u, v \in \real^p$
\begin{flalign}
\mbox{[subgradient inequality]}&& \varphi(v) -\varphi(u) \leq & \langle \varphi'(v), v-u\rangle,&\quad   \label{subgrad}\\
\mbox{[Lipschitz continuity]}&&  \|\varphi'(u) \| \leq & L,&\quad \label{lip}
\end{flalign}
where $\varphi'(v)$ is an element of $\partial \varphi(v)$. 
For that purpose, we introduce the variables
\begin{align*}
	& \textstyle x_\ast \in \argmin_x \varphi(x), \\
	& \delta_i = \varphi(x_i), \quad i=M+1,\dots,N,\ast,\\
	& g_i \in \partial \varphi(x_i), \quad i=M+1,\dots,N,\ast,
\end{align*}
and for ease of notation, we set
\begin{align*}
	& \delta_j = f(x_j), \quad j=1,\dots,M, \\
	& g_j = f'(x_j), \quad j=1,\dots,M.
\end{align*}
We now relax $P_M(h)$ by replacing the function variable $\varphi$ with the new variables
and by introducing constraints that follow from the application of the
subgradient inequality \eqref{subgrad} and the Lipschitz-continuity of $\varphi$ \eqref{lip}
at the points $x_1,\dots,x_N,x_\ast$,
reaching
the following minimax problem
in finite dimension:
\begin{align*}
    &\max_{\substack{g_{M+1},\dots,g_N,g_\ast,x_\ast\in\mathbb{R}^p,\\\delta_{M+1},\dots,\delta_N,\delta_\ast \in\mathbb{R}}}\ \delta_N-\delta_\ast \\
    \text{s.t. }
			& x_{i} = x_0 + \sum_{k=1}^{i-1} h^{(i)}_{1,k} (x_k-x_0) - \sum_{k=1}^{r} h^{(i)}_{2,k} v_k - \sum_{k=M+1}^{i-1} h^{(i)}_{3,k} g_k, \quad i=M+1,\dots,N,\\
			& \delta_i -\delta_j \leq  \langle g_i, x_i-x_j \rangle, \quad i,j=1,\dots,N,\ast, \\
			& \|g_i\|^2 \leq L^2, \quad i=1,\dots,N,\ast\\
			& \|x_\ast-x_0\|^2 \leq R^2.
\end{align*}

Recall that $\delta_j, g_j$ and $x_j$, $j=1,\dots,M$, are given in advance (these are the trial points) and
are considered as the problem's data. 

It appears that this problem (which clearly is not convex) remains nontrivial to tackle. We therefore consider a relaxation obtained by removing some constraints:
\begin{align*}
    P_M^I(h):= &\max_{\substack{g_{M+1},\dots,g_N,x_\ast\in\mathbb{R}^p,\\\delta_{M+1},\dots,\delta_N,\delta_\ast \in\mathbb{R}}}\ \delta_N-\delta_\ast \\
    \text{s.t. }
			& x_{i} = x_0 + \sum_{k=1}^{i-1} h^{(i)}_{1,k} (x_k-x_0) - \sum_{k=1}^{r} h^{(i)}_{2,k} v_k - \sum_{k=M+1}^{i-1} h^{(i)}_{3,k} g_k, \quad i=M+1,\dots,N,\\
			& \delta_i -\delta_j \leq  \langle g_i, x_i-x_j \rangle, \quad i=M+1,\dots,N,\ j=1,\dots,i-1, \\
			& \delta_i - \delta_\ast \leq  \langle g_i, x_i-x_\ast \rangle, \quad i=1,\dots,N,\\
			& \|g_i\|^2 \leq L^2, \quad i=M+1,\dots,N,\\
			& \|x_\ast-x_0\|^2 \leq R^2.
\end{align*}

As before, we denote the problem $\min_h P_M^I(h)$ by $(P_M^I)$, and we have
$$
\val(P_M) = \min_{h} P_M(h) \leq \min_{h} P_M^I(h)=\val(P_M^I).
$$
Our first main objective is now to derive a tractable convex minimization problem which is an upper-bound for the minimax problem $(P^I_M)$.
The first step in that direction is the derivation of a semidefinite programming relaxation of the inner maximization problem $P_M^I(h)$.
At this juncture, the reader might naturally be wondering why we do not derive directly a dual problem of the inner maximization
to reduce our minimax problem to a minimization problem. It turns out that the SDP relaxation derived below enjoys
a fundamental monotonicity property (see Lemma~\ref{L:descent}),
which will play a crucial role in the proof of the main complexity result Theorem~\ref{T:main}.


\subsection{Relaxing The Inner Maximization Problem to an SDP}\label{S:sdprelaxation}
We proceed by performing a semidefinite relaxation on $P_M^I(h)$.
Let $X\in \mathbb{S}^{1+r+N-M}$ be 
\[
	X = \footnotesize
	\begin{pmatrix}
		\langle x_\ast-x_0,x_\ast-x_0\rangle	&	\langle x_\ast-x_0,v_1\rangle	&	\cdots	&	\langle x_\ast-x_0,v_r\rangle	&	\langle x_\ast-x_0,g_{M+1}\rangle	&	\cdots	&	\langle x_\ast-x_0,g_N\rangle	\\
		\langle v_1, x_\ast-x_0\rangle			&	\langle v_1,v_1\rangle			&	\cdots	&	\langle v_1,v_r\rangle &	\langle v_1,g_{M+1}\rangle	&	\cdots	&	\langle v_1,g_N\rangle \\
		\vdots									&	\vdots							&	\ddots	&	\vdots					&	\vdots						&	\ddots	&	\vdots				\\
		\langle v_r, x_\ast-x_0\rangle			& 	\langle v_r, v_1\rangle			&	\cdots	&	\langle v_r,v_r\rangle	&	\langle v_r,g_{M+1}\rangle	&	\cdots	&	\langle v_r,g_N\rangle \\
		\langle g_{M+1},x_\ast-x_0\rangle			&	\langle g_{M+1}, v_1\rangle		&	\dots	&	\langle g_{M+1}, v_r\rangle		&	\langle g_{M+1},g_{M+1} \rangle	&		\cdots	&	\langle g_{M+1},g_N\rangle	\\
		\vdots							&		\vdots					& \ddots	&	\vdots	&	\vdots						&	\ddots								&	\vdots		\\
		\langle g_N,x_\ast-x_0\rangle			&	\langle g_{N}, v_1\rangle		&	\dots	&	\langle g_N, v_r\rangle		&	\langle g_N,g_{M+1}\rangle	& \cdots 	&	\langle g_N,g_N\rangle	
	\end{pmatrix},
\]
and let $\bv_i, \bg_i, \bx_i\in \mathbb{R}^{1+r+N-M}$ be such that
\begin{equation}\label{D:boldletters}
\begin{aligned}
	& \bv_i = e_{1+i}, \quad i=1,\dots,r,\\
	& \bg_i = \begin{cases} \sum_{k=1}^r \langle g_i, v_k\rangle \bv_k, & i=1,\dots,M,\\
							e_{1+r+i-M}^T,  & i=M+1,\dots,N, \end{cases}\\
	& \bx_i = \begin{cases} \sum_{k=1}^r \langle x_i-x_0, v_k\rangle \bv_k, &i=1,\dots,M, \\
							\sum_{k=1}^{i-1} h^{(i)}_{1,k} \bx_k - \sum_{k=1}^{r} h^{(i)}_{2,k} \bv_k - \sum_{k=M+1}^{i-1} h^{(i)}_{3,k} \bg_k, &i=M+1,\dots,N,\\
							e_1,				& i =\ast,
							\end{cases}
\end{aligned}
\end{equation}
then it is straightforward to verify that the following identities hold
\begin{equation}\label{E:identities}
\begin{aligned}
	& \bv_i^T X \bv_j = \langle v_i, v_j\rangle, \quad i, j=1,\dots,r,\\
	& \bg_i^T X \bg_j  = \langle g_i, g_j\rangle, \quad i,j=1,\dots,N, \\
	& \bg_i^T X \bx_j = \langle g_i, x_j-x_0\rangle, \quad i=1,\dots,N,\ j=1,\dots,N,*, \\
	& \bx_i^T X \bx_j = \langle x_i-x_0, x_j-x_0\rangle, \quad i,j=1,\dots,N,*.
\end{aligned}
\end{equation}
Now, by using \eqref{E:identities} in $P_M^I(h)$
and by relaxing the definition of $X$ to $\bv_i^T X \bv_j = \langle v_i, v_j\rangle$ and $X\succeq 0$,
we reach the following SDP:
\begin{align*}
    P_M^{II}(h):= \max_{\substack{X\in\mathbb{S}^{1+r+N-M}, \\\delta_{i},\delta_\ast \in\mathbb{R}}}\ & \delta_{N}-\delta_\ast\\
    \text{s.t. }\quad
			& \delta_{i} -\delta_j\leq  \bg_i^T X (\bx_i-\bx_j), \quad i=M+1,\dots,N,\ j=1,\dots,i-1,\\
			& \delta_i -\delta_\ast\leq  \bg_i^T X (\bx_i-\bx_\ast), \quad i=1,\dots,N,\\
			& \bg_i^T X \bg_i \leq L^2, \quad i=M+1,\dots,N,\\
			& \bx_\ast^T X \bx_\ast \leq R^2,\\
			& \bv_i^T X \bv_j = \langle v_i, v_j\rangle, \quad i, j=1,\dots,r,\\
 			& X\succeq 0.
\end{align*}

Again, we define
\[
	(P_M^{II}) \quad \min_h P_M^{II}(h),
\]
and we have
\[
\val(P_M) \leq \val(P_M^I) \leq \val(P_M^{II}).
\]
Note that $P_M^{II}(h)$ depends on the value of $h$ through the vectors $\bx_i$,
therefore $(P_M^{II})$ involves bilinear terms in its optimization variables.

\subsection{Transforming the Minimax SDP to a Minimization Problem}\label{S:takingdual}

To transform the minimax problem $(P_M^{II})$ into a minimization problem, we use duality. More precisely, as shown below, by replacing
$P_M^{II}(h)$ with its Lagrangian-dual, we reach a nonconvex (bilinear) semidefinite minimization problem whose optimal value coincides with that of $(P_M^{II})$.

\begin{lemma}\label{L:bilsdp} Let $P_M^{III}(h)$ be defined by
\begin{align*}
    P_M^{III}(h) := \min_{a, b, c,d, \Phi} &\ \sum_{i=M+1}^N \sum_{j=1}^M a_{i,j} \delta_j+\sum_{i=1}^M b_i (\langle g_i, x_i-x_0\rangle-\delta_i)+L^2\! \sum_{i=M+1}^N c_i + R^2 d + \sum_{i=1}^r \Phi_{i,i}\\
    \text{s.t.}\quad
			& -\sum_{i=M+1}^{N}\left(\sum_{j=1}^{i-1} a_{i,j} (\bx_i-\bx_j) + b_i \bx_i\right) \bg_i^T  +\sum_{i=1}^{N} b_i \bx_\ast \bg_i^T \\&\qquad +\sum_{i=M+1}^N c_i \bg_i \bg_i^T + d \bx_\ast \bx_\ast^T + \sum_{i, j=1}^r \Phi_{i,j} \bv_i \bv_j^T \succeq 0, \\
			&\ (a,b)\in \Lambda,\ a_{i,j}\geq 0,\ b_i\geq 0,\ c_i\geq 0,\ d\geq 0,
\end{align*}
where $$ \Lambda =\{(a,b):\;  \sum_{j=1}^{N-1} a_{N,j} + b_N = 1,\; \sum_{j=1}^N b_j = 1,\; \sum_{j=i+1}^N a_{j,i}- \sum_{j=1}^{i-1} a_{i,j}=b_i, \quad i=M+1,\dots,N-1\}.$$
Then $P^{II}_M(h)=P^{III}_M(h)$ for all $h$.
\end{lemma}

\begin{proof} We attach the dual variables to each of the constraints in $P_M^{II}(h)$ as follows:
\begin{align*}
			a_{i,j}\in \mathbb{R}_+:\ & \delta_{i} -\delta_j\leq  \bg_i^T X (\bx_i-\bx_j), \quad i=M+1,\dots,N,\ j=1,\dots,i-1,\\
			b_i\in \mathbb{R}_+:\ & \delta_i -\delta_\ast\leq  \bg_i^T X (\bx_i-\bx_\ast), \quad i=1,\dots,N,\\
			c_i\in \mathbb{R}_+:\ & \bg_i^T X \bg_i \leq L^2, \quad i=M+1,\dots,N,\\
			d\in \mathbb{R}_+:\ & \bx_\ast^T X \bx_\ast \leq R^2, \\
			\Phi_{i,j}\in \mathbb{R}:\ & \bv_i^T X \bv_j = \langle v_i, v_j\rangle, \quad i, j=1,\dots,r.
\end{align*}
Recalling that
$\delta_i$ and $\bg_i^T X \bx_i=\langle g_i,x_i-x_0\rangle$ are fixed for $i=1,\dots,M$,
and that the set $\{v_1,\dots,v_r\}$ is orthonormal,
the Lagrangian for this maximization problem
is given by
\begin{eqnarray*}
L(X,\delta;a,b,c,d,\Phi)&=& \delta_{N}-\delta_\ast+\sum_{i=M+1}^N D_i \delta_i + D_\ast \delta_\ast + \trace( X W ) + \mathcal{C},\\
& \equiv & L_1(\delta;a,b) + \trace (X W) + \mathcal{C},
\end{eqnarray*}
with
\begin{eqnarray*}
D_i & = & -\sum_{j=1}^{i-1} a_{i,j} + \sum_{j=i+1}^N a_{j,i} - b_i, \quad i=M+1,\dots,N, \\
D_\ast &= & \sum_{j=1}^N b_j,
\end{eqnarray*}
\begin{eqnarray*}
W &=&\sum_{i=M+1}^{N}\sum_{j=1}^{i-1} a_{i,j} (\bx_i-\bx_j)\bg_i^T +\sum_{i=M+1}^{N} b_i \bx_i \bg_i^T  -\sum_{i=1}^{N} b_i \bx_\ast \bg_i^T -\sum_{i=M+1}^N c_i \bg_i \bg_i^T \\&&- d \bx_\ast \bx_\ast^T - \sum_{i, j=1}^r \Phi_{i,j} \bv_i \bv_j^T,\\
\mathcal{C} &=& \sum_{i=M+1}^N \sum_{j=1}^M a_{i,j} \delta_j+\sum_{i=1}^M b_i (\langle g_i, x_i-x_0\rangle-\delta_i)+L^2\! \sum_{i=M+1}^N c_i + R^2 d+ \sum_{i=1}^r \Phi_{i,i}.
\end{eqnarray*}
The dual objective function is then defined by
$$
H(a,b,c,d, \Phi)=\max_{\delta, X} L(X, \delta; a,b,c,d, \Phi)=\mathcal{C}+ \max_{\delta}L_1(\delta;a,b) + \max_{X\succeq 0} \trace( X W).
$$
Since $L_1(\delta;a,b)$ is linear in the variables $\delta_i$, $i=M+1,\dots,N,\ast$,
the first maximization problem is equal to zero whenever
\begin{align*}
\left\{
\begin{aligned}
			& D_i = -\sum_{j=1}^{i-1} a_{i,j} + \sum_{j=i+1}^N a_{j,i} - b_i = 0, \quad i=M+1,\dots,N-1, \\
			& 1 + D_N = 1 - \sum_{j=1}^{N-1} a_{N,j} - b_N = 0,\\
			& -1 + D_\ast = -1 + \sum_{j=1}^N b_j = 0,
\end{aligned}
\right.
\end{align*}
i.e., when $(a,b)\in \Lambda$,
and is equal to infinity otherwise. Likewise, the second maximization is equal to zero whenever $W \preceq 0$, and is equal to infinity otherwise.
Therefore, the dual problem of $P_M^{II}(h)$ reads as
\[
	\min_{a,b,c,d, \Phi }H(a,b,c,d, \Phi)=\min_{a, b, c,d, \Phi} \{ \mathcal{C}: W\preceq 0,\ (a,b)\in \Lambda,\ a_{i,j}\geq 0,\ b_i\geq 0,\ c_i\geq 0,\ d\geq 0\},
\]
which reduces to the minimization problem $P_M^{III}(h)$.

Now, as a consequence of weak duality for the pair of problems $(P^{II}_M(h))$--$(P^{III}_M(h))$ it immediately follows that
$$\val (P_M^{II})=\min_h P_M^{II}(h) \leq \min_{h} P_M^{III}(h)=\val (P_M^{III}).$$ Furthermore, observing that
$P_M^{II}(h)$ is feasible and that $P_M^{III}(h)$
is strictly feasible (since the elements in the diagonal of the SDP constraint,
i.e., $c_i$, $d$, and $\Phi_{i,i}$, can be chosen to be arbitrarily large), then by invoking the conic duality theorem~\cite[Theorem~2.4.1]{ben2001lectures},
strong duality holds, i.e., $P_M^{II}(h) = P_M^{III}(h)$, and the proof is complete.
\end{proof}
As an immediate consequence, we have
\[
	\val (P_M^{III}):= \min_h P_M^{III}(h)  = \min_h P_M^{II}(h)  = \val (P_M^{II}).
\]

\subsection{A Tight Convex SDP Relaxation for $(P^{III}_M)$}\label{S:tightconvex}

At this stage, the minimization problem $(P^{III}_M)$ we have just derived remains a nonconvex (bilinear) problem.
Indeed, as noted above, the vectors $\bx_i$ depend on the optimization variable $h$,
hence the terms $a_{i,j} (\bx_i-\bx_j)$ and $b_i \bx_i$ in $(P_M^{III})$ are bilinear.
We will now show that it is possible to derive a {\em tight convex relaxation} for this problem.
This will be achieved through two main steps as follows.

\paragraph{Step I: Linearizing the bilinear SDP.}
As just noted,  the terms $a_{i,j} (\bx_i-\bx_j)$ and $b_i \bx_i$ in $(P_M^{III})$ are bilinear.
Here we linearize these terms by introducing new variables $\xi_{i,j}$ and $\psi_{i,j}$ such that
\begin{equation}\label{E:linearize}
	-\left(\sum_{j=1}^{i-1} a_{i,j} (\bx_i-\bx_j)+b_i \bx_i\right) = \sum_{j=1}^{r} \xi_{i,j} \bv_j + \sum_{j=M+1}^{i-1} \psi_{i,j} \bg_j, \quad i=M+1,\dots,N.
\end{equation}
Using \eqref{E:linearize} to eliminate the bilinear terms in $(P_M^{III})$ yields the following linear SDP:
\begin{align*}
    (P_M^{IV})\ \min_{a, b, c, d, \xi, \psi, \Phi} &\ \sum_{i=M+1}^N \sum_{j=1}^M a_{i,j} \delta_j+\sum_{i=1}^M b_i (\langle g_i, x_i-x_0\rangle-\delta_i)+L^2\! \sum_{i=M+1}^N c_i + R^2 d + \sum_{i=1}^r \Phi_{i,i} \\
    \text{s.t.}\quad
			& \sum_{i=M+1}^N \left( \sum_{j=1}^{r} \xi_{i,j} \bv_j + \sum_{j=M+1}^{i-1} \psi_{i,j} \bg_j \right)\bg_i^T  +\sum_{i=1}^{N} b_i \bx_\ast \bg_i^T \\&\qquad +\sum_{i=M+1}^N c_i \bg_i \bg_i^T + d \bx_\ast \bx_\ast^T + \sum_{i,j=1}^r \Phi_{i,j} \bv_i \bv_j^T \succeq 0, \\
			& (a,b)\in \Lambda,\ a_{i,j}\geq 0,\ b_i\geq 0,\ c_i\geq 0,\ d\geq 0.
\end{align*}

Since any feasible point for $(P_M^{III})$ can be transformed using \eqref{E:linearize}
to a feasible point for $(P_M^{IV})$ without affecting the objective value,
we have
\begin{equation}\label{revers}
\val (P_M^{IV})\leq \val(P_M^{III}).
\end{equation}
As a first step in establishing inequality in the other direction (and therefore equality), 
we introduce the following lemma,
which shows how to  recover
a feasible point for $(P_M^{III})$
from a feasible point for $(P_M^{IV})$
provided that the point satisfies a certain condition.
\begin{lemma}\label{L:delin}
Suppose that $(a, b, c, d, \xi, \psi, \Phi)$ is feasible for $(P_M^{IV})$
and satisfies
\begin{equation}\label{E:nonzerolin}
	\sum_{j=1}^{i-1} a_{i,j}+b_i=0 \Rightarrow \xi_{i,k}=\psi_{i,k}=0,\ \forall k<i.
\end{equation}
Then by taking\footnote{In order to avoid overly numerous special cases,
we adopt the convention $\frac{0}{0}=0$.}
\begin{align*}
		h^{(i)}_{1,k} = \frac{a_{i,k}}{\sum_{j=1}^{i-1} a_{i,j}+ b_i}, \quad
		h^{(i)}_{2,k} = \frac{\xi_{i,k}}{\sum_{j=1}^{i-1} a_{i,j}+ b_i}, \quad
		h^{(i)}_{3,k} = \frac{\psi_{i,k}}{\sum_{j=1}^{i-1} a_{i,j}+ b_i},
\end{align*}
we get that $(h, a, b, c, d, \Phi)$
is feasible for $(P_M^{III})$ and attains the same objective value.
\end{lemma}
\begin{proof}
It is enough to verify that the linearization identity \eqref{E:linearize} is satisfied for the chosen values of $h$.
First, when $\sum_{j=1}^{i-1} a_{i,j}+b_i=0$, recalling that we use the convention $\frac{0}{0}=0$, the identity~\eqref{E:linearize} follows immediately from the assumption~\eqref{E:nonzerolin}
and since the step sizes are all zeros.
Suppose $\sum_{j=1}^{i-1} a_{i,j}+b_i>0$, then
substituting the term $\bx_i$ in \eqref{E:linearize} by its definition in~\eqref{D:boldletters}, we get
that for every $i=M+1,\dots,N$
\begin{align*}
	& -\left(\sum_{j=1}^{i-1} a_{i,j} (\bx_i-\bx_j)+b_i \bx_i\right) =\sum_{j=1}^{i-1} a_{i,j} \bx_j-\left(\sum_{j=1}^{i-1} a_{i,j} +b_i\right) \bx_i\\
	&=\sum_{j=1}^{i-1} a_{i,j} \bx_j-\left(\sum_{j=1}^{i-1} a_{i,j} +b_i\right) \left(\sum_{k=1}^{i-1} h^{(i)}_{1,k} \bx_k - \sum_{k=1}^{r} h^{(i)}_{2,k} \bv_k - \sum_{k=M+1}^{i-1} h^{(i)}_{3,k} \bg_k\right)\\
	&= \sum_{j=1}^{r} \xi_{i,j} \bv_j + \sum_{j=M+1}^{i-1} \psi_{i,j} \bg_j,
\end{align*}
where the last equality follows from the choice of $h$.
\end{proof}
In order to establish that the relaxation performed in this step is indeed tight,
it is enough to show that condition~\eqref{E:nonzerolin} holds for an optimal solution of $(P_M^{IV})$.
However,
before we can show 
how to obtain an optimal solution with the required property,
we need to perform an additional transformation on the problem, which in turn
will also be very useful when deriving the steps of Algorithm~KLM in Section~\ref{S:algklm}.



\paragraph{Step II: Simplifying the problem $(P_M^{IV})$.}
An equivalent and significantly simpler form of problem
$(P_M^{IV})$ can be derived using the matrix completion theorem.

Consider the PSD constraint in $(P_M^{IV})$ in its explicit form,
\[\footnotesize
	Q:=\begin{pmatrix}	d	&	\frac{1}{2} \sum_{k=1}^M b_k \langle g_k, v_1\rangle	& \cdots 	&	\frac{1}{2} \sum_{k=1}^M b_k \langle g_k, v_r\rangle & \frac{1}{2}b_{M+1}	&	\cdots &	\frac{1}{2}b_{N}\\	
				\frac{1}{2} \sum_{k=1}^M b_k \langle g_k, v_1\rangle 	&	\Phi_{1,1}	&	\cdots	&	\Phi_{1,r}	& \frac{1}{2} \xi_{M+1,1}	&	\cdots	&	\frac{1}{2} \xi_{N,1}\\
				\vdots 													&	\vdots		&	\ddots	&	\vdots		& \vdots					&	\ddots	&	\vdots\\
				\frac{1}{2} \sum_{k=1}^M b_k \langle g_k, v_r\rangle	&	\Phi_{r,1}	&	\cdots	&	\Phi_{r,r}	& \frac{1}{2} \xi_{M+1,r}	&	\cdots	& \frac{1}{2} \xi_{N,r}\\
				\frac{1}{2}b_{M+1}										&	\frac{1}{2} \xi_{M+1,1}	&	\cdots	&	\frac{1}{2} \xi_{M+1,r}		&	\\
				\vdots													&	\vdots					&	\ddots	&	\vdots						&		&	R\\
				\frac{1}{2}b_N											&	\frac{1}{2} \xi_{N,1}	&	\cdots	&	\frac{1}{2} \xi_{N,r}		&	\\
	\end{pmatrix} \succeq 0,
\]
with
\[
	R:= \begin{pmatrix}
		c_{M+1}	&	\frac{1}{2} \psi_{M+2,M+1}	&	\cdots	&	\frac{1}{2} \psi_{N,M+1}\\
		\frac{1}{2} \psi_{M+2,M+1}	&	c_{M+2}		&	\ddots &	\vdots\\
		\vdots		&	\ddots	&	\ddots	&	\frac{1}{2} \psi_{N,N-1}\\
		\frac{1}{2} \psi_{N,M+1}	&	\cdots	&	\frac{1}{2} \psi_{N,N-1}	&	c_N
	\end{pmatrix}.
\]
Then by the properties of PSD matrices, $Q \succeq 0$
implies that the principal minors of $Q$ are also PSD.
As a result, we get that
the problem
\begin{align*}
    (P_M^{V})\quad \min_{a, b, c, d, \Phi} &\ \sum_{i=M+1}^N \sum_{j=1}^M a_{i,j} \delta_j+\sum_{i=1}^M b_i (\langle g_i, x_i-x_0\rangle-\delta_i)+L^2\! \sum_{i=M+1}^N c_i + R^2 d + \sum_{i=1}^r \Phi_{i,i} \\
    \text{s.t.}\quad
			& \begin{pmatrix}	d	&	\frac{1}{2}\sum_{k=1}^M b_k \langle g_k, v_i\rangle \\ \frac{1}{2}\sum_{k=1}^M b_k \langle g_k, v_i\rangle	&	\Phi_{i,i} \end{pmatrix} \succeq 0,\quad  i=1,\dots,r,\\
			& \begin{pmatrix}	d	&	\frac{1}{2} b_i \\ \frac{1}{2} b_i	&	c_i \end{pmatrix} \succeq 0,\quad  i=M+1,\dots,N,\\
			& (a,b)\in \Lambda,\ a_{i,j}\geq 0,\ b_i\geq 0,\ c_i\geq 0,\ d\geq 0,
\end{align*}
obtained by replacing $Q\succeq 0$
with constraints of the form $Q_{\{1,i\}\times\{1,i\}}\succeq 0$,
is a relaxation of $(P_M^{IV})$, and thus $\val(P_M^{V})\leq \val(P_M^{IV})$.
As we shall prove below, it turns out that this relaxation is, in fact tight, i.e., $\val(P_M^{V})= \val(P_M^{IV})$.
To establish this result, we need   the following lemma, which is a special case of the matrix completion theorem~\cite{grone1984positive}.

\begin{lemma}\label{L:matrixcompletion}
Suppose $q_{1,i}=q_{i,1}$ and $q_{i,i}$ ($i=1,\dots,n$) are numbers such that
\[
	\begin{pmatrix}
	q_{1,1}	&	q_{1,i}\\
	q_{i,1}	&	q_{i,i}
	\end{pmatrix}\succeq 0, \quad i=2,\dots,n.
\]
Then by taking
\begin{equation}\label{E:matrixcompletion}
	q_{i,j} = q_{j,i} =
		\frac{q_{1,i} q_{1,j}}{q_{1,1}},	
\end{equation}
for $i,j=2,\dots,n$, $i\neq j$,
we get that the $n\times n$ matrix $(q_{i,j})$ is positive semidefinite.
\end{lemma}
\begin{proof}
Suppose $q_{1,1}=0$, then by the properties of PSD matrices, $q_{1,i}$ and $q_{i,1}$ must also be equal to zero.
By adopting the convention $\frac{0}{0}=0$, we get that $q_{i,j} = q_{j,i} = 0$ for $i,j=2,\dots,n$, hence the matrix $(q_{i,j})$ is diagonal and the result is trivial.

Now assume $q_{1,1}>0$ and
let $\gamma=(q_{1,1},\dots,q_{1,n})^T$,
then the claim follows immediately by observing that the matrix $(q_{i,j})$ is the sum of
the positive semidefinite rank-one matrix $q_{1,1}^{-1} \gamma \gamma^T$ and
the nonnegative diagonal matrix $\diag(0,q_{2,2}-q_{1,2}^2/q_{1,1},\dots,q_{n,n}-q_{1,n}^2/q_{1,1})$.
\end{proof}

The promised tightness of the relaxation performed in this step now follows.
\begin{corollary}\label{C:tightnessV}
Suppose $(a^\ast, b^\ast,c^\ast,d^\ast,\Phi_{i,i}^\ast)$ is an optimal solution for $(P_M^{V})$,
then taking
\begin{equation}
	\begin{aligned}
	& \Phi_{i,j}^\ast =
		\frac{\sum_{k=1}^M b_k^\ast \langle g_k, v_i\rangle \sum_{k=1}^M b_k^\ast \langle g_k, v_j\rangle}{2 d^\ast}, \quad i,j=1,\dots,r,\ i\neq j,\\
	& \xi_{i,j}^\ast =
		\frac{b_i^\ast\sum_{k=1}^M b_k^\ast \langle g_k, v_j\rangle}{2 d^\ast},  \quad i=M+1,\dots,N,\ j=1,\dots,r,\\
	& \psi_{i,j}^\ast =
		\frac{ b_i^\ast b_j^\ast}{2d^\ast},  \qquad i=M+1,\dots,N,\ j=M+1,\dots,i-1.
	\end{aligned}\label{E:optimalxipsi}
\end{equation}
we get that $(a^\ast, b^\ast,c^\ast,d^\ast,\xi^\ast,\psi^\ast,\Phi^\ast)$ is
an optimal solution for $(P_M^{IV})$.
In particular, we have $\val (P_M^{IV})= \val (P_M^{V})$.
\end{corollary}
\begin{proof}
Observing that the minors of $Q$ selected in $(P_M^{V})$ have the same form as
in the premise of Lemma~\ref{L:matrixcompletion}
with $n=1+r+(N-M)$,
\begin{align*}
	& q_{1,1}=d,\\
	& q_{1+i,1+i}=\Phi_{i,i}, \quad i=1,\dots,r,\\
	& q_{1+r+i,1+r+i}=c_i, \quad i=M+1,\dots,N,\\
	& \textstyle q_{1,1+i}=q_{1+i,1}= \frac{1}{2} \sum_{k=1}^M b_k \langle g_k, v_1\rangle, \quad i=1,\dots,r,\\
	& \textstyle q_{1,1+r+i}=q_{1+r+i,1}=\frac{1}{2}b_i, \quad i=M+1,\dots,N,
\end{align*}
we get that using the choice~\eqref{E:optimalxipsi}, the relations~\eqref{E:matrixcompletion} are satisfied,
hence $Q$ is PSD and the first constraint in $(P_M^{IV})$ is satisfied for $(a^\ast, b^\ast,c^\ast,d^\ast,\xi^\ast,\psi^\ast,\Phi^\ast)$.
Now, examining $(P_M^{IV})$, we see that the variables $\Phi_{i,j}$ for $i\neq j$,
$\xi_{i,j}$ and $\psi_{i,j}$, do not participate in constraints beside the first constraint or in the objective,
hence we conclude that $(a^\ast, b^\ast,c^\ast,d^\ast,\xi^\ast,\psi^\ast,\Phi^\ast)$ is feasible for $(P_M^{IV})$ and furthermore $\val (P_M^{IV}) \leq \val (P_M^{V})$.
Since we have already established that $\val (P_M^{V})\leq \val (P_M^{IV})$, the proof is complete.
\end{proof}

Another consequence of Lemma~\ref{L:matrixcompletion} is the tightness of the relaxation performed in Step~I, allowing us to complete our main goal of this section.

\begin{corollary}\label{C:tightnessIII} The following equality holds:
\[
	\val (P_M^{IV})=\val (P_M^{III}).
\]
\end{corollary}
\begin{proof} Let $(a^\ast, b^\ast,c^\ast,d^\ast,\Phi_{i,i}^\ast)$ be an optimal solution for $(P_M^{V})$.
Then from Corollary~\ref{C:tightnessV} we get that
by taking $\xi^\ast$, $\psi^\ast$, and $\Phi^\ast$ as in~\eqref{E:optimalxipsi}, the point
$(a^\ast, b^\ast,c^\ast,d^\ast,\xi^\ast,\psi^\ast,\Phi^\ast)$ is optimal for $(P_M^{IV})$.
Observing that from~\eqref{E:optimalxipsi} we get that $b_i^\ast= 0$ implies $\xi_{i,j}^\ast=0$ and $\psi_{i,j}^\ast=0$,
then it follows that assumption~\eqref{E:nonzerolin} is
satisfied, 
hence Lemma~\ref{L:delin} is applicable on $(a^\ast, b^\ast,c^\ast,d^\ast,\xi^\ast,\psi^\ast,\Phi^\ast)$.
As a result, 
the optimal value of $(P_M^{IV})$ is attainable by $(P_M^{III})$, 
and since we also have $\val (P_M^{IV})\leq \val (P_M^{III})$ (see~\eqref{revers}), we conclude that $\val (P_M^{III})=\val (P_M^{IV})$,
proving the desired claim.
\end{proof}

\paragraph{Summary.}
To summarize the results up to this point, by performing a series of relaxations and transformations on $(P_M)$, which defined the worst-case absolute inaccuracy at $x_N$, we obtained a sequence of problems $(P^I_M)$--$(P^{V}_M)$ that satisfy
\[
	\val (P_M) \leq \val (P^I_M) \leq \val (P^{II}_M) = \dots = \val (P^{V}_M),
\]
where  the solution of $(P^{V}_{M})$ provides a tractable upper bound. We are now left with our second main goal, namely
to derive the steps of algorithm KLM as defined through problem $(B_M)$ in Section \ref{S:kelleyvariant}.

\section{Derivation of Algorithm KLM}\label{S:algklm}

At first glance, problem $(P^V_{M})$ does not seem to share much resemblance to problem $(B_M)$.
We now proceed to show that this convex SDP problem admits a pleasant equivalent convex minimization reformulation over a simplex in $\real^{M+1}$,
and that this representation is, in fact, the dual of problem $(B_M)$.

\subsection{Reducing $(P^V_{M})$ to a Convex Minimization Problem Over the Unit Simplex}

The form $(P_M^{V})$ allows us to derive analytical optimal solutions to some of the optimization variables.
First, for any fixed $(a,b,d)$, it is easy to see that the minimization with respect to $\Phi$ and $c$
yields the optimal solutions
\begin{align}
	& \Phi^\ast_{i,i} = \frac{(\sum_{k=1}^M b_k \langle g_k,v_i\rangle)^2}{4 d}, \quad i=1,\dots,r, \label{E:optimalphi}\\
	& c^\ast_i = \frac{b_i^2}{4d}, \quad i=M+1,\dots,N. \label{E:optimalc}
\end{align}
Therefore, recalling that $\{v_1,\dots,v_r\}$ is an orthonormal set that spans $g_1,\dots,g_M$, we get
\begin{align*}
	\sum_{j=1}^r \Phi^\ast_{j,j} = \sum_{j=1}^r \frac{(\sum_{i=1}^M b_i \langle g_i,v_j\rangle)^2}{4 d}= \frac{\| \sum_{j=1}^r \sum_{i=1}^M b_i  \langle g_i,v_j\rangle v_j\|^2}{4 d}= \frac{\|\sum_{i=1}^M b_i g_i\|^2}{4 d},
\end{align*}
and $(P_M^{V})$ becomes
\begin{align*}
	\min_{a, b, d} &\ \sum_{i=M+1}^N \sum_{j=1}^M a_{i,j} \delta_j+\sum_{i=1}^M b_i (\langle g_i, x_i-x_0\rangle-\delta_i) + R^2 d + \frac{L^2 \sum_{k=M+1}^N b_i^2 + \|\sum_{i=1}^M b_i  g_i\|^2}{4 d}\\
    \text{s.t.}
			&\ \ (a,b)\in \Lambda,\ a_{i,j}\geq 0,\ b_i\geq 0,\ d\geq 0.
\end{align*}
Next, observe that for any fixed $(a,b)$ the minimization with respect to $d$ is also immediate and
yields
\begin{equation}\label{E:optimald}
	d^\ast = \frac{\sqrt{\textstyle \|{\sum_{i=1}^M b_i g_i}\|^2+ L^2 \sum_{i=M+1}^N b_i^2 }}{2R}.
\end{equation}
Plugging this in the last form of the problem, we reach
\begin{align}
	\begin{aligned}
    \min_{a, b} & \sum_{i=M+1}^N \sum_{j=1}^M a_{i,j} \delta_j+\sum_{i=1}^M b_i (\langle g_i, x_i-x_0\rangle-\delta_i)+ R\sqrt{\textstyle \|{\sum_{i=1}^M b_i  g_i}\|^2+ L^2 \sum_{i=M+1}^N b_i^2} \\
    \text{s.t.}
			&\ (a,b)\in \Lambda,\ a_{i,j}\geq 0,\ b_i\geq 0.
	\end{aligned}\label{P:vb}
\end{align}
Now, fixing $b$, the above minimization problem is a linear program in the variable $a$,
which, as shown by the following lemma, can be solved analytically.
\begin{lemma}\label{L:optimala}
Suppose $b\in \Delta_N$, where $\Delta_N$ denotes the $N$-dimensional unit simplex, i.e., $\Delta_N:=\{b\in\mathbb{R}^N: \sum_{i=1}^N b_i =1, b_i\geq 0\}$.
Then,
\begin{align*}
	\min_{a} \left\{ \sum_{i=M+1}^N \sum_{j=1}^M a_{i,j} \delta_j: (a,b)\in \Lambda,\ a_{i,j}\geq 0\right\}= \sum_{i=1}^M b_i \delta_m,
\end{align*}
where an optimal solution is given by
\begin{equation}\label{E:optimala}
	a_{i,j}^\ast = \begin{cases}
		\sum_{i=1}^M b_i	&	i=N,\ j=m, \\
		b_j,						& 	i=N,\ j\in\{M+1,\dots,N-1\},\\
		0,						&	\text{otherwise,}
	\end{cases}
\end{equation}
with
\begin{equation}\label{E:optimalm}
	m\in \argmin_{1\leq i\leq M} \delta_i.
\end{equation}
\end{lemma}
\begin{proof}
Observe that if we fix $a_{i,j}$ for $j>M$, the constraints in $\Lambda$ have the form
\[
	\sum_{j=1}^M a_{i,j} = \text{constant}, \quad i=M+1,\dots,N,
\]
and we get that the problem is separable into $N-M$ minimization problems over a simplex.
This implies that the optimal solution can be attained by setting $a_{i,j}^\ast=0$
for all $j\in \{1,\dots,M\}\setminus\{m\}$ (i.e., for all indices except for an index for which $\delta_j$ is minimal).
Using this assignment,
the objective now reads
\[
\sum_{i=M+1}^N a_{i,m} \delta_m,
\]
and $\Lambda$ is reduced to (taking into account all variables):
\begin{align*}
		& - a_{i,m}-\sum_{j=M+1}^{i-1} a_{i,j} + \sum_{k=i+1}^N a_{k,i} - b_i = 0, \quad i=M+1,\dots,N-1, \\
		& 1 - a_{N,m}-\sum_{j=M+1}^{N-1} a_{N,j} - b_N = 0,\\
		& - 1 + \sum_{i=1}^N b_i = 0.
\end{align*}
Summing up the constraints in $\Lambda$, we get
\begin{align*}
	 \sum_{i=M+1}^N a_{i,m} &= -\sum_{i=M+1}^{N-1}\left(\sum_{j=M+1}^{i-1} a_{i,j} - \sum_{k=i+1}^N a_{k,i} \right) -\sum_{j=M+1}^{N-1} a_{N,j} +\sum_{i=1}^M b_i\\
	& = \sum_{i=M+1}^{N-1}\sum_{k=i+1}^N a_{k,i} -\sum_{i=M+1}^{N} \sum_{j=M+1}^{i-1} a_{i,j}  + \sum_{i=1}^M b_i
	 = \sum_{i=1}^M b_i,
\end{align*}
which means that the optimal value for the objective is $\sum_{i=1}^M b_i \delta_m$.
It is now straightforward to verify that the given solution~\eqref{E:optimala} is feasible and attains the optimal value of the problem,
hence the proof is complete.
\end{proof}

Invoking Lemma~\ref{L:optimala}, we can write problem \eqref{P:vb} in the following form:
\begin{equation}\label{reduce}
    \min_{b\in \Delta_{N}}\quad  \sum_{i=1}^M b_i (\langle g_i, x_i-x_0\rangle+\delta_m-\delta_i) + R\sqrt{\textstyle \|{\sum_{i=1}^M b_i  g_i}\|^2+ L^2 \sum_{i=M+1}^N b_i^2 }.
\end{equation}
To complete this step,
note that if $b^\ast$ is an optimal solution of the last convex problem
then optimality conditions imply that we must have $b_{M+1}^\ast=\dots=b_N^\ast$.
We can therefore assume, without affecting the optimal value of the problem, that $b_{M+1}=\dots=b_N$,
hence,
by introducing the variable $\beta = \sum_{i=M+1}^N b_i$,
we get
\begin{equation}\label{E:optimalb}
	b_{M+1}=\dots=b_N = \frac{\beta}{N-M},
\end{equation}
and hence
\[
	\sum_{i=M+1}^N b_i^2 = (N-M)b_N^2 = (N-M)\left(\frac{\beta}{N-M}\right)^2=\frac{\beta^2}{N-M}.
\]	
Therefore, using this in \eqref{reduce}, we have shown
\begin{proposition}\label{P:equiv} The convex SDP problem  $(P_M^{V})$ admits
the equivalent convex minimization formulation
\begin{align*}
    (P_M^{VI}) \quad \min_{(b_1,\dots,b_M,\beta)\in \Delta_{M+1}}\quad & \sum_{i=1}^M b_i (\langle x_i-x_0,g_i\rangle+\delta_m-\delta_i) + R\sqrt{\textstyle \|{\sum_{i=1}^M b_i  g_i}\|^2+ \frac{L^2 \beta^2}{N-M} },
\end{align*}
and we have $\val (P_M^{V})= \val(P_M^{VI})$.
\end{proposition}

\subsection{Completing the Derivation of KLM}
We are now ready to complete the main goal of this section, namely the derivation of Algorithm KLM.
Indeed, as shown below, it turns out that the convex problem $(P_M^{VI})$ is nothing else but a dual representation of problem $(B_M)$
defined in Section~\ref{S:kelleyvariant}.
More precisely, we establish that strong duality holds for the pair of convex problems $(P_M^{VI})$--$(B_M)$.
Furthermore, as a by-product, we derive
the desired output of the method as described in Section~\ref{S:kelleyvariant}.
To prove this result, we first recall the following elementary fact.
\begin{lemma}\label{a:cs} Let $D \in \mathbb{S}^{l}_{++}, q \in \real^l$ and $R >0$ be given. Then,
\begin{equation}\label{cs}
\max_{u \in \real^l} \{ \langle q, u \rangle:\; u^TDu \leq R^2\}= R \| D^{-1/2}q\|\; \mbox{with optimal}\; u^*=R \frac{D^{-1}q}{\|D^{-1/2}q\|}.
\end{equation}
\end{lemma}
\begin{proof} The claim is an immediate consequence of Cauchy-Schwartz inequality and can also be derived by simple calculus. \end{proof}

The first main result of this section now follows.
\begin{proposition}\label{strongd}
Let $(B_M)$ be the problem defined in Section~\ref{S:kelleyvariant}, where for $M=0$, we take
\[
	(B_0)\quad \max_{y\in \mathbb{R}^p,\ \zeta \in \mathbb{R}} \left\{ L\zeta : \|y-x_0\|^2+N \zeta^2\leq R^2 \right\}.
\]
Then the pair of convex problems $(P_M^{VI})$--$(B_M)$ are dual to each other, and strong duality holds\footnote{Note that since both problems admit a compact feasible set, attainment of both values is warranted.}, i.e.,
$\val (P_M^{VI})=\val (B_M)$.
Moreover, given an optimal solution  $(b_1^\ast,\dots,b_M^\ast,\beta^\ast)$ for $(P_M^{VI})$, an optimal solution $(y^*,\zeta^*)$ for $(B_M)$ is recovered via
\begin{equation}\label{E:optimaly}
y^\ast = x_0 - \frac{1}{2d^\ast}\sum_{j=1}^M b_j^\ast g_j\; \mbox{ and }\;  \zeta^\ast = \frac{L\beta^\ast}{2(N-M)d^\ast},
\end{equation}
with 
\[
	d^\ast=\frac{\sqrt{\textstyle \|{\sum_{i=1}^M b_i^\ast g_i}\|^2+ \frac{L^2 (\beta^\ast)^2}{N-M} }}{2R}.
\]
\end{proposition}

\begin{proof} Invoking Lemma~\ref{a:cs} with $u := (y-x_0,\zeta)$ and $q:=(-\sum_{i=1}^{M}b_ig_i, L\beta)$,
both in $\real^p \times \real$, and with the block diagonal matrix $D:= [I_p; (N-M)^{-1}] \in \mathbb{S}^{p+1}_{++}$, it easily follows that problem $(P_M^{VI})$ reads as the convex-concave minimax problem:
\begin{align*}
   V_*:= \min_{(b_1,\dots,b_M,\beta)\in \Delta_{M+1}} \max_{\|y-x_0\|^2+(N-M)\zeta^2\leq R^2}\  &  \sum_{i=1}^M b_i(\langle x_i-y,g_i\rangle+\delta_m-\delta_i) + \beta L \zeta.
\end{align*}
Applying the minimax theorem~\cite{fan1953minimax}, we can reverse the min-max operations, and hence by using the simple fact $\min_{\alpha \in \Delta_l}\sum_{i=1}^{l}\alpha_iv_i=\min_{1\leq i \leq l}v_i$ it follows that
\begin{align*}
   V_*= \max_{\|y-x_0\|^2+(N-M)\zeta^2\leq R^2} \min \  & \left\{ \delta_m-\delta_1+\langle x_1-y,g_1\rangle,\dots, \delta_m-\delta_M+\langle x_M-y,g_M\rangle, L\zeta\right\},
\end{align*}
which is an obvious equivalent reformulation of the problem $(B_M)$.
This establishes  the strong duality claim $\val (P_M^{VI})=\val (B_M)$.
Furthermore, if $(b^*, \beta^*)\in \Delta_{M+1}$ is optimal for $(P_M^{VI})$,  again thanks to Lemma \ref{a:cs},
(with $(q,u, D)$ as defined above), one immediately recovers an optimal solution $(y^*,\zeta^*)$ of $(B_M)$ as given in \eqref{E:optimaly}
and the proof is completed.
\end{proof}


As we now show, Proposition~\ref{strongd} paves the way to determine the iterative steps of Algorithm KLM.
For that purpose, we first derive an expression for $x_{M+1},\dots,x_N$ in terms an optimal solution $(b_1^\ast,\dots,b_M^\ast,\beta^\ast)$ for $(P_M^{VI})$.
First, recall that $(a^\ast, b^\ast,c^\ast,d^\ast,\xi^\ast,\psi^\ast,\Phi^\ast)$
with $a^\ast$, $b^\ast$, $c^\ast$, $\Phi^\ast_{i,i}$, $d^\ast$, $\xi^\ast$, $\psi^\ast$, and $\Phi^\ast$
defined according to \eqref{E:optimala}, \eqref{E:optimalb}, \eqref{E:optimalc}, \eqref{E:optimalphi}, and \eqref{E:optimalxipsi},
is optimal for $(P_M^{IV})$ and
satisfies the assumption~\eqref{E:nonzerolin}.
Thus, as a result of Lemma~\ref{L:delin}
and the definition of the sequence $x_i$ in~\eqref{E:mainalg},
the corresponding sequence $x_{M+1},\dots,x_N$
can be found via the rule
\begin{equation}\label{E:xstepiv}
	 x_i= x_0+\frac{1}{\sum_{j=1}^{i-1} a_{i,j}^\ast+ b_i^\ast} \left( \sum_{j=1}^{i-1} a_{i,j}^\ast (x_j-x_0)  - \sum_{j=1}^{r} \xi^\ast_{i,j} v_j - \sum_{j=M+1}^{i-1} \psi^\ast_{i,j} g_j\right).
\end{equation}
From definitions of $\xi^\ast$ and $\psi^\ast$ in~\eqref{E:optimalxipsi} we get that
\[
	\sum_{j=1}^r \xi_{i,j}^\ast v_j = \frac{b_i^\ast}{2d^\ast} \sum_{j=1}^r \sum_{k=1}^M b_k^\ast \langle g_k, v_j\rangle v_j = \frac{b_i^\ast}{2d^\ast}\sum_{k=1}^M b_k^\ast g_k,
\]
and
\[
	\sum_{j=1}^{r} \xi_{i,j}^\ast v_k + \sum_{j=M+1}^{i-1} \psi_{i,j}^\ast g_j = \frac{b_i^\ast}{2d^\ast}\sum_{j=1}^{i-1} b_j^\ast g_j,
\]
which, together with \eqref{E:xstepiv},
yields an expression for $x_i$ that is independent of $\xi_{i,j}^\ast$ and $\psi_{i,j}^\ast$:
\begin{equation}\label{E:eliminatedxi}
	 x_i= \frac{1}{\sum_{j=1}^{i-1} a_{i,j}^\ast+ b_i^\ast} \left( \sum_{j=1}^{i-1} a_{i,j}^\ast x_j+  b_i^\ast \left(x_0- \frac{1}{2d^\ast}\sum_{j=1}^{i-1} b_j^\ast g_j \right)\right), \quad i=M+1,\dots,N.
\end{equation}
Now, using the definition of $a^\ast$ from~\eqref{E:optimala}, we reach the expression
\begin{equation*}
	 x_i = \left\{
		\begin{aligned}
		&x_0- \frac{1}{2d^\ast}\sum_{j=1}^{i-1} b_j^\ast g_j,					& \hspace{-50pt}	i=M+1,\dots,N-1,\\
		&\sum_{j=1}^M b_j^\ast x_m + \sum_{j=M+1}^{N-1} b_j^\ast x_j +b_N^\ast \left(x_0- \frac{1}{2d^\ast}\sum_{j=1}^{N-1} b_j^\ast g_j \right), & i=N,
		\end{aligned}
		\right.
\end{equation*}
where
$m$ as in~\eqref{E:optimalm}.

This rule can be written in a more convenient form using a solution to the pair of convex problems $(P_M^{VI})$--$(B_M)$. For that, note that
by writing $x_i$ in terms of $x_{i-1}$,  breaking the computation of the last step, $x_N$ into two parts $x_N$ and $\bar x_N$,
and applying \eqref{E:optimaly} of Proposition~\ref{strongd}, we obtain
\begin{equation}\label{E:easyalg}
	\begin{aligned}
	 &x_i = \begin{cases}
		\displaystyle x_0- \frac{1}{2d^\ast}\sum_{j=1}^{M} b_j^\ast g_j = y^*,					& i=M+1,\\
		\displaystyle x_{i-1} - \frac{\beta^\ast}{2(N-M)d^\ast} g_{i-1}=x_{i-1}- \frac{\zeta^*}{L} g_{i-1}, & i=M+2,\dots,N,
		\end{cases}\\
	&\bar x_N = (1-\beta^\ast) x_m + \frac{\beta^\ast}{N-M} \sum_{j=M+1}^N x_j,
	\end{aligned}
\end{equation}
which is precisely the output of Algorithm~KLM
after performing a ``\hardstepname'' step followed by $N-M-1$ ``easy'' steps.

Note that for $M=0$
the analytical solution $\beta^\ast=1$, $d^\ast = \frac{L}{2R\sqrt{N}}$ and $\eta^\ast = \frac{R}{\sqrt{N}}$
can be easily established.
The calculation above then yields
\begin{equation}\label{E:initstepalg}
	\begin{aligned}
	 & x_i = \begin{cases}
		\displaystyle x_0,					& i=1,\\
		\displaystyle x_{i-1}- \frac{R}{L \sqrt{N}} g_{i-1}, & i=2,\dots,N,
		\end{cases} \quad
	 \bar x_N = \frac{1}{N} \sum_{j=M+1}^N x_j,
	\end{aligned}
\end{equation}
which is the output of Algorithm~KLM when no ``\hardstepname'' steps are taken.
As an immediate result we obtain that if Algorithm~KLM takes no ``\hardstepname'' steps
it achieves
\begin{equation}\label{E:easybound}
f(\bar x_N)-f^\ast \leq \val (P_{0}^{II}) = \val (P_{0}^{VI}) = \frac{LR}{\sqrt{N}},
\end{equation}
where the last equality follow by observing that $\beta^\ast=1$.

\section{The Rate of Convergence: Proof of Theorem~\ref{T:main}}\label{S:proofmt}
Before we proceed with the proof of Theorem~\ref{T:main}, we need the following lemma,
which establishes that the optimal value of $(P_M^{II})$ 
is non-increasing during the run of the method.
\begin{lemma}\label{L:descent}
Let $l\in \mathbb{N}$ be such that $M+l\leq N$ and
suppose $x_{M+1},\dots,x_{M+l}$ satisfy the recursion~\eqref{E:mainalg} with $h=\bar h$,
where $\bar h$ is optimal for the outer minimization problem in $(P_M^{II})$.
Then $\val (P_{M+l}^{II}) \leq \val (P_M^{II})$.
\end{lemma}
\begin{proof}
%
Denote by $\hat h$ the steps sizes in $\bar h$ which correspond to the last $N-M-l$ steps $x_{M+l+1},\dots,x_N$
(i.e., $\hat h^{(i)}_{j,k} = \bar h^{(i)}_{j,k}$ for $i=M+l+1,\dots,N$),
and let $(\hat X,\hat \delta)$ be optimal for
the inner maximization problem in $(P_{M+l}^{II})$ when fixing $h=\hat h$.
We proceed by constructing a matrix $\bar X$ and a vector $\bar \delta$ such that $(\bar h; \bar X, \bar \delta)$ is feasible to $(P_M^{II})$ and achieves the same objective value as $(\hat h; \hat X, \hat \delta)$
achieves for $(P_{M+l}^{II})$.

As we've seen in the previous section, the optimal value of $(P_{M}^{II})$
does not depend on the specific choice of $r$ or the set $\{v_1,\dots,v_r\}$,
we can therefore assume without loss of generality that both $(P_{M}^{II})$ and $(P_{M+l}^{II})$
are expressed using the same value of $r$ and the same set of vectors $v_i$ (e.g., we can choose $r=d$ and take $\{v_i\}$ as the canonical basis).

Denote by $\bar \bv_i$, $\bar \bg_i$ and $\bar \bx_i$ the vectors $\bv_i$, $\bg_i$ and $\bx_i$ as defined for $(P_M^{II})$ in \eqref{D:boldletters},
and let $\hat \bv_i$, $\hat \bg_i$ and $\hat \bx_i$ be the vectors $\bv_i$, $\bg_i$ and $\bx_i$ that correspond to $(P_{M+l}^{II})$, i.e.,
\begin{equation*}
\begin{aligned}
	& \bar \bv_i = e_{1+i}, \quad i=1,\dots,r,\\
	& \bar \bg_i = \begin{cases} \sum_{k=1}^r \langle g_i, v_k\rangle \bv_k, & i=1,\dots,M,\\
							e_{1+r+i-M}^T,  & i=M+1,\dots,N, \end{cases}\\
	& \bar \bx_i = \begin{cases} \sum_{k=1}^r \langle x_i-x_0, v_k\rangle \bv_k, &i=1,\dots,M, \\
							\sum_{k=1}^{i-1} h^{(i)}_{1,k} \bx_k - \sum_{k=1}^{r} h^{(i)}_{2,k} \bv_k - \sum_{k=M+1}^{i-1} h^{(i)}_{3,k} \bg_k, &i=M+1,\dots,N,\\
							e_1,				& i =\ast,
							\end{cases}
\end{aligned}
\end{equation*}
and
\begin{align*}
	& \hat \bv_i = e_{i+1}, \quad i=1,\dots,r,\\
	& \hat \bg_i = \begin{cases} \sum_{k=1}^r \langle g_i, v_k\rangle \hat\bv_k, & i=1,\dots,M+l,\\
							e_{1+r+i-M}^T,  & i=M+l+1,\dots,N, \end{cases}\\
	& \hat \bx_i = \begin{cases} \sum_{k=1}^r \langle x_i-x_0, v_k\rangle \hat\bv_k, &i=1,\dots,M+l, \\
							\sum_{k=1}^{i-1} h^{(i)}_{1,k} \hat\bx_k - \sum_{k=1}^{r} h^{(i)}_{2,k} \hat\bv_k - \sum_{k=M+1}^{i-1} h^{(i)}_{3,k} \hat\bg_k, &i=M+l+1,\dots,N,\\
							e_1						& i =\ast.
							\end{cases}
\end{align*}
Now, by taking $V$ as the $(1+r+N-M-l)\times(1+r+N-M)$ matrix
\[
	V=(\hat \bx_\ast,\hat \bv_1,\dots,\hat \bv_r,\hat \bg_{M+1},\dots,\hat \bg_N),
\]
it follows from the construction above that
\begin{align*}
	& \hat \bv_i = V \bar \bv_i, \quad i=1,\dots,r, \\
	& \hat \bg_i = V \bar \bg_i, \quad i=1,\dots,N, \\
	& \hat \bx_i = V \bar \bx_i, \quad i=1,\dots,N,\ast.
\end{align*}
Hence, by setting
\begin{align*}
	\bar X =& V^T \hat X V,\\
	\bar \delta_i = &
		\begin{cases}
			f(x_i),	& i=M+1,\dots,M+l,\\
			\hat \delta_i,	& i=M+l+1,\dots,N,\ast,
		\end{cases}
\end{align*}
we get that the equalities
\begin{align*}
	& \bar \bg_i^T \bar X \bar \bg_j = \hat \bg_i^T \hat X \hat \bg_j, \quad i,j=1,\dots,N, \\
	& \bar \bg_i^T \bar X \bar \bx_j = \hat \bg_i^T \hat X \hat \bx_j, \quad i=1,\dots,N,\ j=1,\dots,N,*.
\end{align*}
are satisfied, and therefore
$(\bar h; \bar X, \bar \delta)$ satisfies all the constraints in $(P_M^{II})$ that also appear in $(P_{M+l}^{II})$.
Note, however, that $(P_M^{II})$ includes some additional constraints that do not appear in $(P_{M+l}^{II})$, namely
\[
	\bar\delta_{i} -\bar\delta_j\leq  \bar\bg_i^T X (\bar\bx_i-\bar\bx_j),
\]
for $i=M+1,\dots,M+l-1$, and $j=1,\dots,i-1$, and
\[
	\bar\bg_i^T X \bar\bg_i \leq L^2,
\]
for $i=M+1,\dots,M+l-1$.
Nevertheless,
since for $i,j\leq M+l$ the values of $\bar\delta_{i}$, $\bar \bg_i^T \bar X \bar \bg_j$ and $\bar \bg_i^T \bar X \bar \bx_j$
originate from the convex function $f$, i.e.,
\begin{align*}
	& \bar \bg_i^T \bar X \bar \bg_j = \hat \bg_i^T \hat X \hat \bg_j=\langle f'(x_i), f'(x_j)\rangle, \quad i,j=1,\dots,M+l, \\
	& \bar \bg_i^T \bar X \bar \bx_j = \hat \bg_i^T \hat X \hat \bx_j=\langle f'(x_i), x_j\rangle, \quad i=1,\dots,M+l,\ j=1,\dots,M+l,
\end{align*}
we immediately get from the subgradient inequality
and the Lipschitz-continuity of $f$ that these additional constraints hold.
We conclude that
$(\bar h; \bar X, \bar \delta)$ is feasible for $(P_M^{II})$ and attains the same objective value
as does $(\hat h; \hat X, \hat \delta)$ for $(P_{M+l}^{II})$.

For a feasible point $(h; X, \delta)$, denote by $P_M^{II}(h; X, \delta)$ the value of the objective in $(P_M^{II})$ at the given point,
then we have just shown that $P_{M+l}^{II}(\hat h; \hat X, \hat \delta) = P_M^{II}(\bar h; \bar X, \bar \delta)$.
As an immediate consequence, we get
\[
	\val (P_{M+l}^{II}) \leq P_{M+l}^{II}(\hat h; \hat X, \hat \delta) = P_M^{II}(\bar h; \bar X, \bar \delta) \leq \val (P_M^{II}),
\]
where the first inequality follow since $(\hat X, \hat \delta)$ is optimal
for
the inner maximization problem in $(P_{M+l}^{II})$
and the last inequality follows since
$\bar h$ is optimal for the outer minimization problem in $(P_M^{II})$.
\end{proof}

We are now ready to give the proof of Theorem~\ref{T:main}.
\begin{proof}(Theorem~\ref{T:main}.)
We begin by considering the case where no ``\hardstepname'' step was taken. As noted at the end of the previous section,
the chosen parameters at the initialization of Algorithm~KLM correspond to the solution of the outer minimization problem in $(P_0^{II})$,
thereby giving the desired bound~\eqref{E:easybound}.

Now suppose that at least one ``\hardstepname'' step was taken.
Recalling that $s_k$ is the index of the last ``\hardstepname'' step, then
by the definition of the ``easy'' steps, the sequence
$x_{s_k+1},\dots,x_N,\bar x_N$ satisfies \eqref{E:easyalg}, where $y^\ast$, $\zeta^\ast$ and $\beta^\ast$ are
given by a solution of $(B_{s_k})$.
Let $\bar h$ be the vector of step sizes in \eqref{E:mainalg} that matches $x_{s_k+1},\dots,x_{N-1},\bar x_N$,
then by the construction of $(B_{s_k})$ from $(P_{s_k}^{II})$, we get that $\bar h$ is optimal for $(P_M^{II})$,
i.e., $\val (P_{s_k}^{II})=P_{s_k}^{II}(\bar h)$
(we use $P_{s_k}^{II}(\bar h)$ to denote the optimal value of the inner maximization problem in $(P_{s_k}^{II})$ with $h$ set to $\bar h$).
We therefore have
\[
	f(\bar x_N)-f^\ast \leq P_{s_k}(\bar h) \leq P_{s_k}^{II}(\bar h)=\val (P_{s_k}^{II}),
\]
where the two inequalities follow from the construction of $(P_{s_k}^{II})$.
By an immediate application of Lemma~\ref{L:descent} we get
\[
	f(\bar x_N)-f^\ast \leq \val (P^{II}_{s_k})\leq \dots \leq \val (P^{II}_{s_1}) \leq \val (P^{II}_{0}).
\]

Finally, since we have already established during the construction and analysis of Section~\ref{S:tracub} that
the series of relaxations and transformations preserve the optimal value of the problem,
i.e., $\val (P_{M}^{II})=\dots=\val (P_{M}^{VI})=\val (B_{M})$ for every $M$,
and since $\val (P^{II}_{0})=\val (P_{0}^{VI})={LR}/{\sqrt{N}}$,
the claim immediately follows.
\end{proof}


\section{A Numerical Illustration}\label{S:numerical}
In this section, we illustrate the potential benefit of the ``\hardstepname''~steps
by examining the behavior of the two extreme variants of Algorithm~KLM: the ``pure \hardstepname''
which takes a ``\hardstepname''~step at every iteration, and the ``pure easy'', which takes only ``easy'' steps.
Figure \ref{F:periteration} shows the absolute error, $f(\bar x_N)-f^*$, achieved by the two variants for various values of $N$
on the problem $\min_x \|Ax-b\|_\infty$, for $x\in \mathbb{R}^{100}$, where $A\in \real^{200\times 100}$ is a matrix whose entries
were randomly sampled from the uniform distribution on $[-1,1]$ and $b\in\real^{200}$ is a vector whose elements were taken from the same distribution.
The values of $L$ and $R$ provided to the variants were twice their exact values and
the exact solution was obtained using an interior point algorithm.

\begin{figure}[tp]
			\centerline{\includegraphics[height=64mm]{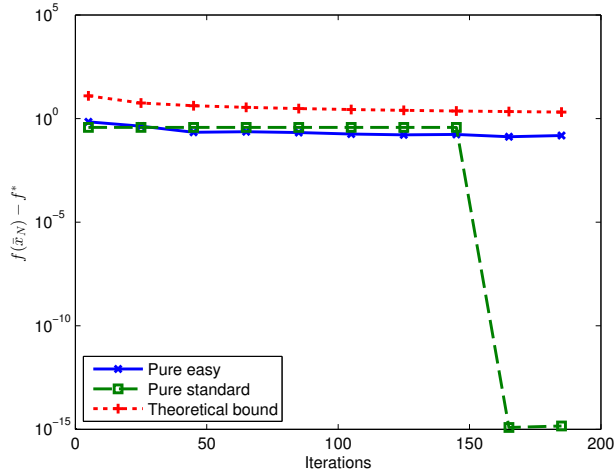}}
			\caption{The absolute error obtained by the ``pure easy'' and ``pure \hardstepname'' variants for different values of $N$.}
			\label{F:periteration}
\end{figure}


As can be seen from this example, although the worst-case guarantee for both variants is identical,
in some cases, the high accuracy obtained by the ``pure \hardstepname'' variant
compensates for the extra cost of its steps.



\section{Concluding Remarks}\label{S:concluding}
Through a constructive approach, we have derived  a new method for non-smooth convex minimization, which is
surprisingly similar to the Kelley method, yet it attains the optimal rate of convergence.
We conclude by briefly discussing how the construction derived in this work can be extended onto some other situations as well, which often arise in nonsmooth optimization schemes/models.

\paragraph{Knowledge of Lower Bound on $f^*$.}
When a lower bound, $\underline{f}$, on $f^\ast$ is known, (e.g., though a dual bound), the constraint
$\underline{f} \leq \varphi^\ast$ can be added to $(P_M)$ and the analysis can continue with only little change.
The resulting method turns out to be nearly the same as the method described above,
where the only change is the introduction of the constraint $\underline{f}\leq t$ to $(B_M)$.
Furthermore, the resulting efficiency estimate remains unchanged.

\paragraph{Inexact Subgradients.}
Another situation is the case where, instead of an exact subgradient, an $\epsilon$-subgradient
$f'(x)\in \partial_\epsilon f(x)$ is available for some given $\epsilon\geq 0$, i.e.,
for any $y$, instead of the usual subgradient inequality, we have
\[
	f(x) -f(y) \leq  \langle f'(x), x-y \rangle+\epsilon.
\]
The use of $\epsilon$-subgradients instead of exact subgradients has some practical advantages,
see e.g., \cite{auslender2004interior, oliveira2013bundle}
and references therein for motivating examples and for some recent work in this setting.
As in the previous case, only minor changes are needed in the analysis we developed,
and the resulting method turns out to be identical to the method presented in Section~\ref{S:kelleyvariant},
except for the first set of constraint in $(B_M)$, which becomes
\[
	f(x_i)+\langle y-x_i,f'(x_i)\rangle-\epsilon\leq t, \quad i=1,\dots,M,
\]
and for the efficiency estimate of the method~\eqref{E:minbound}, which turns out to be
\[
	f(\bar x_N)-f^\ast \leq \val (B_{s_k})+\epsilon\leq \dots \leq \val (B_{s_1})+\epsilon \leq {LR}/{\sqrt{N}}+\epsilon.
\]

\section*{Acknowledgements}
We thank the two referees and the associate editor for their constructive comments and useful suggestions.

\appendix

\section{Appendix: a Tight Lower-Complexity Bound}\label{S:appendix}
In this appendix, we refine the proof from \cite[Section 3.2]{nest-book-04} to obtain a new lower-complexity bound on
the class of nonsmooth, convex, and Lipschitz-continuous functions,
which together with the results discussed above
form a \emph{tight} complexity result for this class of problems.
More precisely, under the setting of \S\ref{SS:klm}, we show that for any first-order method, the worst-case
absolute inaccuracy after $N$ steps cannot be better than $\frac{LR}{\sqrt{N}}$,
which is exactly the bound attained by Algorithm~KLM.

In order to simplify the presentation, and
following \cite[Section 3.2]{nest-book-04}, we restrict our attention to
first-order methods that generate 
sequences that satisfy the following assumption:
\begin{assumption}\label{A:linear}
The sequence $\{x_i\}$ satisfies
\[
	x_i  \in x_1 + \mathrm{span}\{f'(x_1),\dots,f'(x_{i-1})\},
\]
where $f'(x_i)\in \partial f(x_i)$ is obtained by evaluating a first-order oracle at $x_i$.
\end{assumption}
As noted by Nesterov~\cite[Page 59]{nest-book-04}, this assumption is not necessary and can be avoided by some additional reasoning.

The lower-complexity result is stated as follows.
\begin{theorem}
For any $L,R>0$, $N,p\in\mathbb{N}$ with $N\leq p$, and any starting point $x_1\in \mathbb{R}^p$,
there exists a convex and Lipschitz-continuous function $f:\mathbb{R}^p\rightarrow \mathbb{R}$
with Lipschitz constant $L$ and $\|x^\ast_f-x_1\|\leq R$,
and a first-order oracle $\mathcal{O}(x)= (f(x), f'(x))$,
such that
\[
	f(x_N)-f^\ast \geq \frac{LR}{\sqrt{N}}
\]
for all sequences $x_1,\dots,x_N$ 
that satisfies Assumption~\ref{A:linear}.
\end{theorem}
\begin{proof}
The proof proceeds by constructing a ``worst-case'' function, on which
any first-order method that satisfies Assumption~\ref{A:linear} will not be able to improve its initial
objective value during the first $N$ iterations.

Let $f_N:\mathbb{R}^p\rightarrow \mathbb{R}$ and $\bar f_N:\mathbb{R}^p\rightarrow \mathbb{R}$
be defined by
\begin{align*}
	&f_N(x) = \max_{1\leq i \leq N} \langle x, e_i\rangle,  \\
	& \bar f_N(x) =L\max(f_N(x), \|x\|-R(1+N^{-1/2})),
\end{align*}
then it is easy to verify that $\bar f_N$ is Lipschitz-continuous with constant $L$ and that
\[
	\bar f_N^\ast = -\frac{LR}{\sqrt{N}}
\]
is attained for $x^\ast\in \mathbb{R}^p$ such that
\[
	x^\ast = -\frac{R}{\sqrt{N}}\sum_{i=1}^N e_i.
\]
We equip $\bar f_N$ with the oracle $\mathcal{O}_N(x)= (\bar f_N(x), \bar f'_N(x))$ by choosing
$\bar f'_N(x)\in \partial \bar f_N(x)$ according to:
\begin{equation}\label{E:resistingb}
	\bar f'_N(x) = \begin{cases}
			L f'_N(x),	& f_N(x)\geq \|x\|-R(1+N^{-1/2}),\\
			L\frac{x}{\|x\|},	& f_N(x)< \|x\|-R(1+N^{-1/2}),
		\end{cases}
\end{equation}
where
\begin{equation}\label{E:resisting}
	f'_N(x) = e_{i^\ast}, \quad i^\ast = \min \{ i : f_N(x)=\langle x, e_i\rangle \}.
\end{equation}
We also denote
\[
	\mathbb{R}^{i,p} := \{x\in\mathbb{R}^d : \langle x, e_j\rangle=0,\ i+1\leq j\leq p\}.
\]

Now, let $x_1,\dots,x_N$ be a sequence that satisfies Assumption~\ref{A:linear} with $f=\bar f_N$ and the oracle $\mathcal{O}_N$,
where without loss of generality we assume $x_1=0$.
Then $\bar f'_N(x_1) = e_1$ and we get $x_2\in\mathrm{span}\{\bar f'_N(x_1)\}=\mathbb{R}^{1,p}$.
Now, from $\langle x_2,e_2\rangle=\dots=\langle x_2,e_N\rangle=0$,
we get that $\min \{ i : f_N(x)=\langle x, e_i\rangle \}\leq 2$
and it follows by~\eqref{E:resistingb} and~\eqref{E:resisting} that $f'_N(x_2)\in \mathbb{R}^{2,p}$
and $\bar f'_N(x_2)\in \mathbb{R}^{2,p}$.
Hence, we conclude from Assumption~\ref{A:linear} that $x_3 \in \mathrm{span}\{\bar f'_N(x_1),\bar f'_N(x_2)\}\subseteq \mathbb{R}^{2,p}$.
It is straightforward to continue this argument to show that
$x_i \in \mathbb{R}^{i-1,p}$ and $\bar f'_N(x_i)\in \mathbb{R}^{i,p}$ for $i=1,\dots,N$, thus $x_N \in \mathbb{R}^{N-1,p}$.
Finally,
since for every $x\in \mathbb{R}^{N-1,p}$ we have $\bar f_N(x)\geq \langle x,e_N\rangle =0$,
we immediately get
\[
	\bar f_N(x_{N})-\bar f_N^\ast \geq \frac{LR}{\sqrt{N}},
\]
which completes the proof.
\end{proof}


\bibliographystyle{abbrv}
\bibliography{bib}

\end{document}